\documentclass[12pt]{amsart}
\usepackage{graphicx}
\usepackage{amssymb}
\usepackage{amsmath}
\usepackage{amsthm}
\usepackage{amscd}
\usepackage{epsfig}
\usepackage{color}

\newtheorem{thm}{Theorem}[section]

\newtheorem{lem}[thm]{Lemma}
\newtheorem{prop}[thm]{Proposition}
\newtheorem{ques}{Question}

\newtheorem{ex}[thm]{Example}
\newtheorem{rem}[thm]{Remark}
\theoremstyle{definition} \theoremstyle{question}
\newtheorem{defn}[thm]{Definition}
\numberwithin{equation}{section}

\newcommand{\N}{\mathbb N}

\def\R {\Bbb R}
\def\Z {\Bbb Z}

\def \ep {\epsilon}
\def \ra {\rightarrow}

\begin{document}

\title{Lowering topological entropy over subsets revisited}

\date{November 28, 2012}

\author{Wen Huang, Xiangdong Ye and Guohua Zhang}

\address{Wu Wen-Tsun Key Laboratory of Mathematics, USTC, Chinese Academy of Sciences and
Department of Mathematics, University of Science and Technology of
China, Hefei, Anhui, 230026, P.R. China}

\email{wenh@mail.ustc.edu.cn, yexd@ustc.edu.cn}

\address{School of Mathematical Sciences and LMNS, Fudan University, Shanghai 200433, China}

\email{zhanggh@fudan.edu.cn}


\thanks{Huang is supported by NNSF of China (11225105), Fok Ying Tung Education Foundation and
the Fundamental Research Funds for the Central Universities,
Huang+Ye are supported by NNSF of China (11071231), and Zhang is
supported by FANEDD (201018) and NSFC (11271078).}

\begin{abstract}
Let $(X, T)$ be a topological dynamical system. Denote by $h (T, K)$
and $h^B (T, K)$ the covering entropy and dimensional entropy of
$K\subseteq X$, respectively.  $(X, T)$ is called D-{\it lowerable}
(resp. {\it lowerable}) if for each $0\le h\le h (T, X)$ there is a
subset (resp. closed subset) $K_h$ with $h^B (T, K_h)= h$ (resp. $h
(T, K_h)= h$); is called D-{\it hereditarily lowerable} (resp. {\it
hereditarily lowerable}) if each Souslin subset (resp. closed
subset) is D-lowerable (resp. lowerable).

In this paper it is proved that each topological dynamical system is
not only lowerable but also D-lowerable, and each asymptotically
$h$-expansive system is D-hereditarily lowerable. A minimal system
which is lowerable and not hereditarily lowerable is demonstrated.

\end{abstract}

\maketitle


\markboth{lowering topological entropy over subsets (II)}{Wen Huang,
Xiangdong Ye and Guohua Zhang}

\section{Introduction}

This paper is a continuation of the research done in \cite{HYZ1} by
the same authors.

Throughout the paper, by a {\it topological dynamical system}
(t.d.s.) $(X, T)$ we mean a compact metric space $X$ together with a
homeomorphism $T: X\rightarrow X$. Let $(X, T)$ be a t.d.s. and
$K\subseteq X$. Denote by $h (T, K)$ and $h^B (T, K)$ the covering
entropy and dimensional entropy of $K\subseteq X$ introduced in
\cite{B} and \cite{B3} respectively. Motivated by \cite{L1, L2, LW,
SW} in \cite{HYZ1} the authors studied the question if for each
$0\le h\le h (T, X)$ there is a closed subset of $X$ with entropy
$h$. Inspired by \cite[Remark 5.13]{YZ}, in \cite{HYZ1} we call $(X,
T)$
\begin{enumerate}

\item {\it lowerable} if for each $0\le h\le h (T, X)$ there is
a closed $K\subseteq X$ with $h (T, K)= h$;

\item {\it hereditarily
lowerable} if each closed subset is lowerable, i.e. for each closed
$K\subseteq X$ and any $0\le h\le h (T, K)$ there is a closed
$K_h\subseteq K$ with $h (T,K)= h$;

\item {\it hereditarily
uniformly lowerable} if for each closed subset $K\subseteq X$ and
any $0\le h\le h (T, K)$ there is a closed $K_h\subseteq K$ such
that $h (T, K_h)= h$ and $K_h$ has at most one limit point.
\end{enumerate}

\noindent Then the question is divided further into the following
questions in \cite{HYZ1}:

\begin{ques} \label{q1}
Is any t.d.s. lowerable?
\end{ques}

\begin{ques} \label{q2}
Is any t.d.s. hereditarily lowerable?
\end{ques}

\begin{ques} \label{q3}
Is any t.d.s. hereditarily uniformly lowerable?
\end{ques}

We remark that the reason we ask Question \ref{q3} in such a way is
that in \cite{YZ} the authors showed that if $(X,T)$ is a t.d.s. and
$K\subset X$ is a compact infinite subset, then there is a countable
subset $K'\subset K$ (the derived set of which has at most one limit
point) with $h(T,K')=h(T,K)$. In \cite{HYZ1} the authors showed that
each t.d.s. with finite entropy is lowerable, and that a t.d.s. is
hereditarily uniformly lowerable iff it is asymptotically
$h$-expansive. In particular, each hereditarily uniformly lowerable
t.d.s. has finite entropy. Moreover, a principal extension preserves
the lowerable, hereditarily lowerable and hereditarily uniformly
lowerable properties. Though we completely answered Question
\ref{q3}, Question \ref{q1} in the case $h(T,X)=+\infty$ and
Question \ref{q2}  still remain open in \cite{HYZ1}.

Let $X$ be a metric space, the {\it Souslin sets} are the sets of
the form
\begin{equation*}
E= \bigcup_{i_1\in \N, i_2\in \N, \cdots} \ \ \bigcap_{k\in \N} E_{i_1,
\cdots, i_k},
\end{equation*}
where $E_{i_1, \cdots, i_k}$ is a closed set for each finite
sequence $\{i_1, \cdots, i_k\}$ of positive integers. Observe that
each Borel set is Souslin, the pre-image of a Souslin set under a
continuous map is Souslin,  and if the underlying metric spaces are
complete then any continuous image of a Souslin set is Souslin. The
well-known result in fractal geometry \cite{Fal, M} states that (for
the definition of Hausdorff dimension see \cite{Fal, M})

\begin{prop}\label{prop2.1}
Let $K\subseteq \R^n$ be a non-empty Souslin subset. Then for each
$0\le h< \text{dim}_H (K)$ there is a compact subset $K_h$ of $K$
with $\text{dim}_H (K_h)= h$, where $\text{dim}_H (*)$ denotes the
Hausdorff dimension of a subset $*$ in $\R^n$.
\end{prop}

\noindent Inspired by this, for a t.d.s. $(X, T)$ we call it
\begin{enumerate}

\item D-{\it
lowerable} if for each $0\le h\le h (T, X)$ there is a subset $K_h$
with $h^B (T, K_h)= h$;

\item D-{\it hereditarily lowerable} if each Souslin subset is D-lowerable,
i.e. for each Souslin set $K\subseteq X$ and any $0\le h\le h^B (T,
K)$ there is $K_h\subseteq K$ with $h^B (T, K_h)= h$.
\end{enumerate}

\noindent Thus, we have other two additional questions:

\begin{ques} \label{q4}
Is any t.d.s. D-lowerable?
\end{ques}

\begin{ques} \label{q5}
Is any t.d.s. D-hereditarily lowerable?
\end{ques}

We emphasize that, in fact, \cite[Theorem 4.4]{HYZ1} tells
us that each t.d.s. with finite entropy is D-lowerable.

In this paper, we get complete answers to Questions \ref{q1} and
\ref{q4}; and partial answers to Questions \ref{q2} and \ref{q5}
(Question \ref{q3} was answered by \cite[Theorem 7.7]{HYZ1}). Namely, with the
help of a relative version of the well-known Sinai Theorem we prove
that each t.d.s. is not only lowerable but also D-lowerable. We
shall construct a minimal lowerable t.d.s. which is not hereditarily
lowerable. Moreover, we also prove that each asymptotically
$h$-expansive t.d.s. is D-hereditarily lowerable. Whereas, there
remain some interesting questions unsolved. For example, is there a
lowerable t.d.s. with finite entropy which is not hereditarily
lowerable?

The paper is organized as follows. In Section 2 the definitions of
cover entropy and dimensional entropy of subsets are introduced. In
Section 3, a minimal lowerable t.d.s. which is not hereditarily
lowerable is presented.
Then in section 4 it is proved that each t.d.s. is not only
lowerable but also D-lowerable with the help of a relative version
of the well-known Sinai Theorem. In the last section, it is shown
that each asymptotically $h$-expansive t.d.s. is D-hereditarily
lowerable.


\medskip
\noindent {\bf Acknowledgement}: We would like to thank Downarowicz,
Glasner and Weiss for useful discussions. We also would like to thank the referee for the careful reading
and useful comments that resulted in substantial improvements to this paper.

\section{Preliminary}

Let $(X, T)$ be a  t.d.s., $K\subseteq X$ and $\mathcal{W}$ a
collection of subsets of $X$. We shall write $K\succeq \mathcal{W}$
if $K\subseteq W$ for some $W\in \mathcal{W}$ and else $K\nsucceq
\mathcal{W}$. If $\mathcal{W}_1$ is another family of subsets of
$X$, $\mathcal{W}$ is said to be {\it finer} than $\mathcal{W}_1$
(we shall write $\mathcal{W}\succeq \mathcal{W}_1$) when $W\succeq
\mathcal{W}_1$ for each $W\in \mathcal{W}$. We shall say that a
numerical function {\it increases} (resp. {\it decreases}) with
respect to (w.r.t.) a set variable $K$ or a family variable
$\mathcal{W}$ if the value never decreases (resp. increases) when
$K$ is replaced by a set $K_1$ with $K_1\subseteq K$ or when
$\mathcal{W}$ is replaced by a family $\mathcal{W}_1$ with
$\mathcal{W}_1\succeq \mathcal{W}$.

By a {\it cover} of $X$ we mean a finite family of Borel subsets
with union $X$, and a {\it partition} a cover whose elements are
disjoint. Denote by $\mathcal{C}_X$ (resp. $\mathcal{C}^o_X$,
$\mathcal{P}_X$) the set of covers (resp. open covers, partitions).
If $\alpha\in \mathcal{P}_X$ and $x\in X$ then let $\alpha (x)$ be
the element of $\alpha$ containing $x$.

Given $\mathcal{U}_1, \mathcal{U}_2\in \mathcal{C}_X$, set
$\mathcal{U}_1\vee \mathcal{U}_2= \{U_1\cap U_2: U_1\in
\mathcal{U}_1, U_2\in \mathcal{U}_2\}$, obviously $\mathcal{U}_1\vee
\mathcal{U}_2\in \mathcal{C}_X$ and $\mathcal{U}_1\vee
\mathcal{U}_2\succeq \mathcal{U}_1$. $\mathcal{U}_1\succeq
\mathcal{U}_2$ need not imply that $\mathcal{U}_1\vee \mathcal{U}_2=
\mathcal{U}_1$, $\mathcal{U}_1\succeq \mathcal{U}_2$ iff
$\mathcal{U}_1$ is equivalent to $\mathcal{U}_1\vee \mathcal{U}_2$
in the sense that each refines the other. For each $\mathcal{U}\in
\mathcal{C}_X$ and any $m, n\in \mathbb{Z}_+$ with $m\le n$ we set
$\mathcal{U}_m^n= \bigvee_{i= m}^n T^{- i} \mathcal{U}$. Moreover,
if $(X, T)$ is a t.d.s. then let $\text{diam} (K)$ be the diameter
of $K$ and put $||\mathcal{W}||= \sup \{\text{diam} (W): W\in
\mathcal{W}\}$, thus if $\mathcal{U}\in \mathcal{C}^o_X$ then
$\mathcal{U}$ has a Lebesgue number $\lambda> 0$ and so
$\mathcal{W}\succeq \mathcal{U}$ when $||\mathcal{W}||< \lambda$.

\subsection{Covering entropy of subsets}

Let $(X, T)$ be a  t.d.s., $K\subseteq X$ and $\mathcal{U}\in
\mathcal{C}_X$. Set $N (\mathcal{U}, K)$ to be the minimal
cardinality of sub-families $\mathcal{V}\subseteq \mathcal{U}$ with
$\cup \mathcal{V}\supseteq K$, where $\cup \mathcal{V}=
\bigcup_{V\in \mathcal{V}}V$. We write $N(\mathcal{U}, \emptyset)=
1$ by convention. Obviously, $N (\mathcal{U}, T (K))= N (T^{- 1}
\mathcal{U}, K)$. Let
\begin{equation*}
h_\mathcal{U} (T, K)= \limsup_{n\rightarrow +\infty} \frac{1}{n}
\log N (\mathcal{U}_0^{n- 1}, K).
\end{equation*}
Clearly $h_\mathcal{U} (T, K)$ increases w.r.t. $\mathcal{U}$.
Define the {\it covering entropy of $K$} by
$$h (T, K)= \sup_{\mathcal{U}\in \mathcal{C}^o_X} h_\mathcal{U} (T, K),$$
and define the {\it topological entropy of $(X, T)$} by
$h_{\text{top}} (T, X)= h (T, X)$.

Let $(X, T)$ and $(Y, S)$ be  t.d.s.s. We say that $\pi: (X,
T)\rightarrow (Y, S)$ is a {\it factor map} if $\pi$ is a continuous
surjection and $\pi\circ T= S\circ \pi$. It is easy to check that

\begin{prop} \label{090319}
Let $(X, T)$ and $(Y, S)$ be t.d.s.s. Then
\begin{enumerate}

\item $h (T, K)\ge h (S, \pi (K))$ if
$\pi: (X, T)\rightarrow (Y, S)$ is a factor map and $K\subseteq X$;

\item $h (T\times S, X\times Y)= h (T, X)+ h (S, Y)$.
\end{enumerate}
\end{prop}

We may also obtain the cover entropy of
subsets using Bowen's separated and spanning sets (see
\cite[P$_{168-174}$]{Wa}). Let $(X, T)$ be a t.d.s. with $d$ a compatible
metric on $X$. For each $n\in \mathbb{N}$ we define a new metric
$d_n$ on $X$ by
$$d_n (x, y)= \max_{0\le i\le n- 1}d (T^i x, T^i y).$$ Let
$\epsilon> 0$ and $K\subseteq X$. A subset $F$ of $X$ is said to
{\it $(n,\epsilon)$-span $K$ w.r.t. $T$} if for each $ x\in K$,
there is $y\in F$ with $d_n(x,y)\le\epsilon$; a subset $E$ of $K$ is
said to be {\it $(n,\epsilon)$-separated w.r.t. $T$} if $x,y\in E,
x\neq y$ implies $d_n(x,y)>\epsilon$. Let $r_n(d,T,\epsilon,K)$
denote the smallest cardinality of any $(n,\epsilon)$-spanning set
for $K$ w.r.t. $T$ and $s_n(d,T,\epsilon,K)$ denote the largest
cardinality of any $(n,\epsilon)$-separated subset of $K$ w.r.t.
$T$. We write $r_n(d,T,\epsilon,\emptyset)=s_n (d, T, \epsilon,
\emptyset)= 1$ by convention. Put
$$r (d, T, \epsilon, K)= \limsup_{n\rightarrow +\infty} \frac{1}{n}
\log r_n (d, T, \epsilon, K)$$ and
$$ s (d, T, \epsilon, K)= \limsup_{n\rightarrow +\infty} \frac{1}{n}
\log s_n (d, T, \epsilon, K).$$ Then put $$h_* (d, T, K)=
\lim_{\epsilon\rightarrow 0+} r (d, T, \epsilon, K)\ \text{and}\ h^*
(d, T, K)= \lim_{\epsilon\rightarrow 0+} s (d, T, \epsilon, K).$$ It
is  well known that $h_* (d, T, K)= h^* (d, T, K)$ is independent of
the choice of a compatible metric $d$ on the space $X$. Now, if
$\mathcal{U}\in \mathcal{C}^o_X$ has a Lebesgue number $\delta> 0$
then, for any $\delta'\in (0, \frac{\delta}{2})$ and each
$\mathcal{V}\in \mathcal{C}^o_X$ with $||\mathcal{V}||\le \delta'$,
one has
$$N(\mathcal{U}_0^{n- 1}, K)\le r_n (d, T, \delta', K)\le s_n (d, T,
\delta', K)\le N(\mathcal{V}_0^{n- 1}, K)$$ for each $n\in
\mathbb{N}$. So if $\{\mathcal{U}_n\}_{n\in \mathbb{N}}\subseteq
\mathcal{C}^o_X$ satisfies $||\mathcal{U}_n||\rightarrow 0$ as
$n\rightarrow +\infty$ then
\begin{equation*}
h_* (d, T, K)=h^* (d, T, K)=\lim_{n\rightarrow +\infty}
h_{\mathcal{U}_n} (T, K)= h (T, K).
\end{equation*}
In this case, it is obvious that $h (T, \overline{K})= h (T, K)$.

\subsection{Dimensional entropy of subsets}

Now we recall the concept of dimensional entropy introduced and
studied in \cite{B3}.

Let $(X, T)$ be a  t.d.s. and $\mathcal{U}\in \mathcal{C}_X$. For
$K\subseteq X$ let
\[
n_{T, \mathcal{U}} (K)=\left\{
\begin{array}{ll}
0, &\mbox{if $K\nsucceq \mathcal{U}$};\\
+\infty, &\mbox{if $T^i K\succeq \mathcal{U}$ when $i\in \mathbb{Z}_+$};\\
k, &\mbox{$k= \max\{j\in\N: T^i(K)\succeq \mathcal{U}\ \text{when}\
0\le i\le j- 1\}$}.
\end{array}
\right.
\]
For $k\in \mathbb{N}$, we define $\mathfrak{C} (T,\mathcal{U},K, k)$
to be the family of all $\mathcal{E}$, where $\mathcal{E}$ is a
countable family of subsets of $X$ such that $K\subseteq \cup
\mathcal{E}$ and $\mathcal{E}\succeq \mathcal{U}_0^{k-1}$. Then for
each $\lambda\in \mathbb{R}$ set
\begin{equation*}
m_{T, \mathcal{U}} (K, \lambda,k)=\inf_{\mathcal{E}\in \mathfrak{C}
(T, \mathcal{U},K, k)} m (T, \mathcal{U}, \mathcal{E}, \lambda),
\end{equation*}
where $$m (T, \mathcal{U}, \mathcal{E}, \lambda)= \sum_{E\in
\mathcal{E}} e^{-\lambda n_{T, \mathcal{U}} (E)},$$ here, by
convention: $0\cdot \infty= 0$ and
$m_{T,\mathcal{U}}(\emptyset,\lambda, k)= +\infty$ if $\lambda< 0$;
$1$ if $\lambda= 0$; $0$ if $\lambda> 0$. As
$m_{T,\mathcal{U}}(K,\lambda,k)$ is increasing w.r.t. $k$, we can
define
$$m_{T,\mathcal{U}}(K,\lambda)=\lim_{k\rightarrow +\infty}
m_{T, \mathcal{U}} (K, \lambda,k).$$ Notice that
$m_{T,\mathcal{U}}(K,\lambda)\le m_{T,\mathcal{U}}(K,\lambda')$ for
$\lambda\ge \lambda'$ and $m_{T, \mathcal{U}} (K, \lambda)\notin
\{0, +\infty\}$ for at most one $\lambda$ \cite{B3}. We define {\it
the dimensional entropy of $K$ relative to $\mathcal{U}$} by
\begin{equation*}
h^B_\mathcal{U} (T, K)= \inf \{\lambda\in \mathbb{R}: m_{T,
\mathcal{U}} (K, \lambda)= 0\}= \sup \{\lambda \in \mathbb{R}: m_{T,
\mathcal{U}} (K, \lambda)= +\infty\}.
\end{equation*}
The {\it dimensional entropy of $K$} is defined by
$$h^B (T, K)= \sup_{\mathcal{U}\in \mathcal{C}^o_X} h_\mathcal{U}^B (T, K).$$
Note that $h_{\mathcal{U}}^B (T, K)$ increases w.r.t.
$\mathcal{U}\in \mathcal{C}_X$, thus if $(X, T)$ is a t.d.s. and
$\{\mathcal{U}_n\}_{n\in \mathbb{N}}\subseteq \mathcal{C}^o_X$
satisfies $||\mathcal{U}_n||\rightarrow 0$ as $n\rightarrow +\infty$
then $\lim_{n\rightarrow +\infty} h^B_{\mathcal{U}_n} (T, K)=h^B (T,
K)$.

The following result is basic (see \cite[Propositions 1 and 2]{B3}
or \cite[Proposition 2.3]{HYZ1}).

\begin{prop} \label{06.02.28}
Let $(X, T)$ be a t.d.s., $K_1, K_2, \cdots, K\subseteq X$
and $\mathcal{U}\in \mathcal{C}_X$. Then
\begin{enumerate}

\item $h_\mathcal{U} (T, X)= h^B_\mathcal{U} (T, X)$ if $\mathcal{U}
\in \mathcal{C}_X^o$, so $h (T, X)= h^B (T, X)$;

\item $h^B_\mathcal{U} (T, \bigcup_{n\in \mathbb{N}}
K_n)= \sup_{n\in \mathbb{N}} h^B_\mathcal{U} (T, K_n)$, so
$$h^B (T,
\bigcup_{n\in \mathbb{N}} K_n)= \sup_{n\in \mathbb{N}} h^B (T,
K_n);$$

\item for each $m\in \mathbb{N}$ and $i\ge 0$,
$h^B_{T^{-i}\mathcal{U}} (T^m, K)= h^B_\mathcal{U} (T^m, T^iK)$,
so $h^B (T^m, K)= h^B (T^m, T^i K)$;

\item for each $m\in
\mathbb{N}$, $h^B_{\mathcal{U}_0^{m- 1}} (T^m, K)= m h^B_\mathcal{U}
(T, K)$, so $h^B (T^m, K)= m h^B (T, K)$.
\end{enumerate}
\end{prop}

Thus, by Proposition \ref{06.02.28} (2), $h^B (T, E)$ increases
w.r.t.
 $E\subseteq X$ and
if $E\subseteq X$ is a non-empty countable set then $h^B (T, E)=0$.
It is worth mentioning that
\begin{enumerate}

\item
$h_{\mathcal{U}}^B(T,\emptyset)=h_{\mathcal{U}}(T,\emptyset)=0$ for
any  $\mathcal{U}\in \mathcal{C}_X$, and so
$h^B(T,\emptyset)=h(T,\emptyset)=0$;

\item when $\emptyset \neq
K\subseteq X$, one has $h_{\mathcal{U}}(T,K)\ge
h_{\mathcal{U}}^B(T,K)\ge 0$ for any $\mathcal{U}\in \mathcal{C}_X$,
and so
$$h(T,K)\ge h^B(T,K)\ge 0.$$
\end{enumerate}

\subsection{Hausdorff dimension and dimensional entropy}

Let $(X,d)$ be a metric space. We first recall the definition of
Hausdorff dimension of a subset $A\subset X$. Fix $t\ge 0$. For each
$\delta>0$ and subset  $A\subset X$, we define
$$H_d^{t,\delta}(A)=\inf \{ \sum_{i=1}^{+\infty} {\mbox{diam}(U_i)}^t\},$$
where the infimum is taken over all countable covers $\{
U_i:i=1,2,\cdots\}$ of $A$ of diameter not exceeding $\delta$. Since
$H_d^{t,\delta}(A)$ increases as $\delta$ decreases for any
$A\subseteq X$, we can define
$$H_d^t(A)=\lim_{\delta\rightarrow 0} H_d^{t,\delta}(A)
=\sup_{\delta>0} H_d^{t,\delta}(A).$$ The case $H_d^t(A)=+\infty$ is
not excluded. Fix $A\subseteq X$. Since for every $0<\delta\le 1$
the function $t\mapsto H_d^{t,\delta}(A)$ is non-increasing, so
is the function $t\mapsto H_d^{t}(A)$. Moreover, if $0<s<t$,
then for every $\delta>0$
$$H_d^{s,\delta}(A)\ge \delta^{s-t}H_d^{t,\delta}(A)$$
which implies that if $H_d^{t}(A)>0$, then $H_d^s(A)=+\infty$. Thus
there is a unique value $H_d(A)\in [0,+\infty]$, which is called the
{\it Hausdorff dimension} of $A$ with respect to the metric $d$ on
$X$, such that
$$H_d^t(A)=\left\{\begin{array}{ll}
 +\infty, &\quad \hbox{ if} \,\; 0\le t< H_d(A),\\
   0, &\quad \hbox{ if}\,\;
H_d(A)<t<\infty.\end{array}\right.
$$

The Hausdorff dimension is a monotone function of sets, i.e. if
$A\subseteq B$ then $H_d(A)\le H_d(B)$. Moreover if $\{ A_n\}_{n\ge
1}$ is a countable family of subsets of $X$ then
$$H_d(\bigcup\limits_{n=1}^\infty A_n)=\sup\limits_{n\ge 1}
H_d(A_n).$$ Hence if $E\subset X$ is countable then $H_d(E)=0$.

\medskip
In the following we investigate the interrelation of Hausdorff
dimension and dimensional entropy of a set in some specific t.d.s.
Let $(X,T)$ be a t.d.s. with metric $d$. We assume that $T$ is {\it
Lipschitz continuous} with the {\it Lipschitz constant} $L$, i.e.
$d(Tx,Ty)\le Ld(x,y)$ for any $x,y\in X$.

The following result is just \cite[Theorem 1]{Mi-add}.

\begin{lem} \label{nche1} Let $(X,T)$ be a Lipschitz continuous t.d.s. with Lipschitz constant $L>1$
associated to the metric $d$. Then
$$H_d(C)\geq {h^{B}(T, C) \over \log L} $$
for any subset  $C\subseteq X$.
\end{lem}

The following result is \cite[Lemma 5.4]{CHYZ}.
\begin{lem} \label{nche2} Let $(X,T)$ be a t.d.s. with metric $d$.
If there exist $\epsilon>0$ and $L>1$ such that $d(Tx,Ty)\geq
Ld(x,y)$ whenever $d(x,y)<\epsilon$, then
$$H_d(C)\le {h^{B}(T,C)\over\log L}$$ for any subset $C\subseteq
X$.
\end{lem}

\begin{prop}\label{L-exapnsive} Let
$\mathbb{T}=\mathbb{R}/\mathbb{Z}$ be the unit circle of complex
plane with the metric $$d(e^{2\pi i x},e^{2\pi i y})=\inf_{k\in
\mathbb{Z}}|x-y-k|$$ for any $x,y\in \mathbb{R}$. For $m\in
\mathbb{N}$ with $m\ge 2$, $T_m$ is defined by $T_m(z)=z^m$ for
$z\in \mathbb{T}$. Then
\begin{enumerate}
\item $H_d(C)={h^{B}(T_m, C)\over\log m}$ for any
subset $C\subseteq \mathbb{T}$.

\item $(\mathbb{T},T_m)$ is D-hereditarily lowerable.
\end{enumerate}
\end{prop}
\begin{proof} (1) follows from Lemmas \ref{nche1} and \ref{nche2} by
setting $L=m$.  Since the metric $d$ and the Euclidean metric on
$\mathbb{T}$ are Lipschitz equivalent,   (2) follows from (1) and
Proposition \ref{prop2.1}.
\end{proof}

\begin{rem} Proposition \ref{L-exapnsive} (1) appeared firstly in \cite[Proposition III.1]{Fu0}.
\end{rem}

\begin{rem} Recall that in \cite[Remark 5.13]{YZ} the authors showed that if $(X,T)$ is a
t.d.s. and $K\subset X$ is a compact subset, then there is a
countable subset $K'\subset K$ (the derived set of which has at most
one limit point) with $h(T,K')=h(T,K)$. Weiss showed us a proof that
when $(X,T)$ is minimal and $X$ is infinite then there exists a countable subset $K$
with a unique limit point such that $h(T,K)=h(T,X)$. In fact, this
can also be obtained by \cite[Theorems 4.2 and 5.7]{YZ}.
\end{rem}

\section{Negative answers to Question \ref{q2}}

In this section, we shall construct a minimal lowerable t.d.s. which
is not hereditarily lowerable. First we give a lowerable t.d.s.
which is not hereditarily lowerable and then we make it minimal. We
remark that the example we get has infinite entropy, and it is not
hard to construct an example which has infinite entropy and at the
same time is hereditarily lowerable.

\subsection{A general example} First we construct an example (not necessarily minimal) which
is lowerable and not hereditarily lowerable. In the next subsection
we will modify it such that it is minimal. To do this, we need the
following lemma.

\begin{lem} \label{prod0}
Let $(X_i, T_i), i\in \N$ be a t.d.s. and $(Y, S)=
\prod\limits_{i\in \N} (X_i, T_i)$ with $\pi_i$ the projection from
$Y$ to $X_i, i\in \N$. Then
\begin{enumerate}

\item $h (T_1^n, W)= n h (T_1, W)$ for each closed $W\subseteq X_1$
 and any $n\in \N$;

\item $h (T_j, \pi_j K)\le h (S, K)\le \sum_{i\in \N} h(T_i, \pi_i K)$
for each closed $K\subseteq Y$ and any $j\in \N$.
\end{enumerate}
\end{lem}
\begin{proof}
(1) It is known, see for example the proof of \cite[Theorem 7.10
(i)]{Wa}.

(2) Let $K\subseteq Y$ be closed and $j\in \N$. By Proposition
\ref{090319} (1) clearly $h (T_j, \pi_j K)\le h (S, K)$. For the
other direction, without loss of generality we assume
$\text{diam} (X_i)\le 1$ with $d_i$ a compatible metric on $X_i$ for
each $i\in\N$. Let $\rho$ be the metric on $Y$ given by
\begin{equation*}
\rho ((x_1, x_2, \cdots), (y_1, y_2, \cdots))= \sum_{i= 1}^\infty
\frac{d_i (x_i, y_i)}{2^i}.
\end{equation*}
For each $\epsilon> 0$ we can select $N (\epsilon)\in \N$ with
$\sum_{i= N (\epsilon)+ 1}^\infty \frac{1}{2^i}<
\frac{\epsilon}{2}$, thus
\begin{equation*}
r_n (\rho, S, \epsilon, K)\le \prod_{i= 1}^{N (\epsilon)} r_n (d_i,
T_i, \frac{\epsilon}{2}, \pi_i K),
\end{equation*}
which implies
\begin{equation*}
r(\rho, S, \epsilon, K)\le \sum_{i= 1}^{N (\epsilon)} r (d_i, T_i,
\frac{\epsilon}{2}, \pi_i K)\le \sum_{i\in \N} r(d_i, T_i,
\frac{\epsilon}{2}, \pi_i K),
\end{equation*}
and so $h (S, K)\le \sum_{i\in \N} h (T_i, \pi_i K)$. This finishes
our proof.
\end{proof}

Thus, we have

\begin{prop} \label{090507}
Let $(X, T)$ be a t.d.s. with topological entropy finite but
positive. Then $(X^\infty, S)= \prod_{n\in \N} (X, T^n)$ is a
lowerable t.d.s. which is not hereditarily lowerable.
\end{prop}
\begin{proof}
Let $E= \{(x, x, \cdots): x\in X\}\subseteq X^\infty$. Now we claim
that the subset $E$ in $(X^\infty, S)$ is not lowerable (and so
t.d.s. $(X^\infty, S)$ is not hereditarily lowerable) by proving
that each closed subset $K$ of $E$ has either infinite topological
entropy or zero topological entropy.

Let $\pi_i: (X^\infty, S)\rightarrow (X, T^i)$ be the factor map of
the $i$-th projection map, $i\in \N$. We shall prove that if $h (T,
\pi_1 K)> 0$ then $h (S, K)= \infty$ and if $h (T, \pi_1 K)= 0$ then
$h (S, K)= 0$. In fact, if $h (T, \pi_1 K)= 0$ then by Lemma
\ref{prod0} one has
\begin{equation*}
h (S, K)\le \sum_{i\in \N} h (T^i, \pi_i K)= \sum_{i\in \N} i h(T,
\pi_i K)= \sum_{i\in \N} i h(T, \pi_1 K)= 0.
\end{equation*}
Now if $h (T, \pi_1 K)> 0$ then by Lemma \ref{prod0} one has
$$h (S,
K)\ge h(T^n, \pi_n K)= n h (T, \pi_n K)= n h(T, \pi_1 K)$$
for each
$n\in \N$, which implies $h (S, K)= \infty$.

Now we shall finish our proof by claiming that $(X^\infty, S)$ is
lowerable. In fact, for each $0\le h< \infty$ we let $n\in \N$ with
$h (S_n, X^n)> h$, where $S_n= T\times T^2\times \cdots \times T^n$.
Note that \cite[Theorem 4.4]{HYZ1} states that each t.d.s.
with finite entropy must be lowerable, whereas, clearly $(X^n, S_n)$
is a t.d.s. with finite entropy, thus there exists closed
$K_h\subseteq X^n$ with $h (S_n, K_h)= h$. Now for each $x_0\in X$
we put $K^*_h= \{(k_1, \cdots, k_n, x_0, \cdots): (k_1, \cdots,
k_n)\in K_h\}$, it is easy to check that $h (S, K^*_h)= h$, this
ends the proof.
\end{proof}

We should remark that \cite[Theorem 6.1]{HYZ1} states that
\begin{center}
{\it Let $(X,T)$ be a t.d.s. Then, for any compact $K\subseteq X$
with $h (T, K)> 0$,\\ there is a countable infinite compact subset
$K_\infty\subseteq K$ with $h (T, K_\infty)= 0$.}
\end{center}

\noindent Whereas, \cite[Theorem 5.9 and Remark 5.13]{YZ} state that

\begin{center}
{\it Let $(X, T)$ be a t.d.s. Then for any compact $K\subseteq X$
there is\\
 a countable compact subset $K'\subseteq K$ with $h (T,
K')= h (T, K)$.}
\end{center}

\noindent Thus, the above Proposition \ref{090507} tells us that
these are the best results we may obtain. In view of this, we
restate our Question \ref{q2} as

\medskip

\noindent {\bf Question 2'}: {\it Is there a t.d.s. with finite
entropy which is not hereditarily lowerable?}

\medskip

It seems to us that a t.d.s. $(X,T)$ is not hereditarily lowerable
if it has an ergodic measure with infinite entropy.

\subsection{A minimal example} After we finished the construction in
the previous subsection, Glasner asked if there is a minimal t.d.s.
with the property. We will show that the example in the previous
subsection can be made minimal. Recall that a t.d.s. $(Y, T)$ is
called {\it proximal orbit dense}, or a {\it POD} system if $(Y,T)$
is totally minimal and whenever $x,y\in Y$ with $x\not=y$, then for
some $n\not=0$, $T^ny$ is proximal to x. An interesting property of
POD is that (see \cite[Corollary 3.5]{KN}):

\begin{prop}\label{KN}  Let $(X, T)$ be POD and consider integers $k_i\not=0, i = 1, \cdots, n$
with $k_i\not=\pm k_j$ for $i\not=j$. Then \begin{enumerate}
\item $(X\times \cdots \times X, T^{k_1}\times \cdots\times T^{k_n})$ is minimal.
\item The only factors of the flow in (1) are the obvious direct
factors. \end{enumerate} If, in addition, $(X, T)$ is not isomorphic
to $(X, T^{-1})$ then (1) and (2) hold for any $k_i\not=0, i = 1,
\cdots, n$ with $k_i\not=k_j$ for $i\not=j$.
\end{prop}

A special class of POD is
\begin{defn} A system $(X, T)$ is said to be doubly minimal if for all $x\in
X$ and $y\not\in \{T^nx\}_{-\infty}^\infty$, the orbit of $(x,y)$ is
dense in $X\times X$.
\end{defn}

The first example of non-periodic doubly minimal system was constructed in
\cite{King} in the symbolic dynamics.
Doubly minimal systems are
natural in the sense that: any ergodic system with zero entropy has
a uniquely ergodic model which is doubly minimal \cite{Weiss}.
 The
notion of disjointness between two t.d.s. was introduced in
\cite{Fu0} and it is easy to see that two minimal t.d.s. are
disjoint iff the product system is minimal \cite[Proposition 2.5]{AG}.

\medskip

\begin{prop} There is a minimal t.d.s. which is not hereditarily lowerable. \end{prop}
\begin{proof} Let $(Y_1,S_1)$ be the non-periodic double minimal system
constructed in \cite{King}, in particular, $(Y_1, S_1)$ is a
strictly ergodic t.d.s. with finite entropy.   Now let $(Y,S)$ be a
minimal t.d.s. with finite positive entropy, which is an extension
of $(Y_1, S_1)$ with a factor map $\pi$ satisfying that $\{y_1\in
Y_1: \pi^{- 1} (y_1)\ \text{is a singleton}\}$ is a residual subset
of $Y_1$ (see \cite[Theorem 3]{DL} for the existence of such a
t.d.s. $(Y,S)$, as $(Y_1, S_1)$ is a strictly ergodic non-periodic
system with finite entropy). Since both $(Y, S)$ and $(Y_1, S_1)$
are minimal, it is well known that the factor map $\pi$ is almost
1-1 in the sense that the subset $\{y\in Y: \pi^{- 1} (\pi y)\
\text{is a singleton}\}\subset Y$ is also residual.

Observe that each non-periodic doubly minimal system is not only
minimal but also weakly mixing, and so totally minimal. In fact, let
$(X, T)$ be a minimal weakly mixing t.d.s. and $m\in \N$, it is well
known that the system $(X, T^m)$ is weakly mixing and each point of
$(X, T^m)$ is minimal, hence $(X, T^m)$ is a minimal t.d.s. For each
pair $x\not= y\in Y_1$, if $y=S_1^nx$ for some $n\not=0$, then $(x,
S_1^{-n}y)=(x,x)$ is proximal; and if $y\not\in
\{S_1^ix\}_{i=-\infty}^\infty$ (which implies that $S_1^ny\not\in
\{S_1^ix\}_{i=-\infty}^\infty$ for any $n\in\Z$), then $(x,S_1^ny)$
has a dense orbit and hence is also proximal for any $n\in\Z$. Thus
$(Y_1, S_1)$, as a non-periodic doubly minimal system, is a POD
system. For a given $n\in\N$, since $(Y_1,S_1)\times
(Y_1,S_1^2)\times \cdots \times (Y_1,S_1^n)$ is minimal (by
Proposition \ref{KN}) it follows that $(Y,S)\times (Y,S^2)\times
\cdots \times (Y,S^n)$ is also minimal \cite[Proposition 2.5]{AG}
(disjointness is preserved by an almost 1-1 extension \cite[Theorem
2.6]{AG}). This implies that $\prod_{i=1}^\infty (Y,S^i)$ is
minimal, as minimality is preserved by the inverse limit. Then by
Proposition \ref{090507}, we get the conclusion.
\end{proof}

\begin{rem} In \cite{Weiss} the author showed that if each pair in a
t.d.s. $(X,T)$ is positively recurrent under $T\times T$ then the
entropy is zero. This is not true if we replace $T\times T$ by
$T\times T^2$ by the proof of the previous proposition. Note that
the question if recurrence under $T\times T$ implies zero entropy is
still open.
\end{rem}

\subsection{A related result}

Finally, we shall present a result related to the property of
hereditary lowering. First, let's make some preparations (for
details see \cite {D2, HYZ1, Mi1, Z1}).

Let $\pi: (X, T)\rightarrow (Y, S)$ be a factor map between t.d.s.s.
The {\it relative topological entropy of $(X, T)$ w.r.t. $\pi$} is
defined as follows:
\begin{equation*}
h_{\text{top}} (T, X| \pi)= \sup_{y\in Y} h (T, \pi^{- 1} (y))=
\sup_{y\in Y} h^B (T, \pi^{- 1} (y)).
\end{equation*}

Observe that (they are proved respectively as \cite[Theorem
7.3]{HYZ1}, \cite[Theorem 3.3]{CHYZ} and \cite[Theorem 4.2]{Z1})

\begin{prop} \label{0905191311}
Let $\pi: (X, T)\rightarrow (Y, S)$ be a factor map between t.d.s.s.
\begin{enumerate}

\item If $E\subseteq X$ is compact then
$$h (S, \pi (E))\le h (T, E)\le h (S, \pi (E))+ h_{\text{top}} (T, X|
\pi).$$

\item If $K\subseteq X$ then
$$h^B (S, \pi (K))\le h^B (T, K)\le h^B (S, \pi (K))+ h_{\text{top}}
(T, X| \pi).$$
\end{enumerate}
\end{prop}

Let $(X, T)$ be a t.d.s. and $\mathcal{U}_1, \mathcal{U}_2\in
\mathcal{C}^o_X$. Put $N (\mathcal{U}_1| \mathcal{U}_2)= \max \{N
(\mathcal{U}_1, U_2): U_2\in \mathcal{U}_2\}$. Then, for each $m,
n\in \mathbb{N}$,
\begin{eqnarray*}
N ((\mathcal{U}_1)_0^{m+ n- 1}| (\mathcal{U}_2)_0^{m+ n- 1})&\le & N
((\mathcal{U}_1)_0^{n- 1}| (\mathcal{U}_2)_0^{m+ n- 1}) N
((\mathcal{U}_1)_n^{m+ n- 1}| (\mathcal{U}_2)_0^{m+ n- 1}) \\
&\le & N ((\mathcal{U}_1)_0^{n- 1}| (\mathcal{U}_2)_0^{n- 1}) N
((\mathcal{U}_1)_n^{m+ n- 1}| (\mathcal{U}_2)_n^{m+ n- 1}) \\
&= & N ((\mathcal{U}_1)_0^{n- 1}| (\mathcal{U}_2)_0^{n- 1}) N
((\mathcal{U}_1)_0^{m- 1}| (\mathcal{U}_2)_0^{m- 1}),
\end{eqnarray*}
i.e. $\{\log N ((\mathcal{U}_1)_0^{n- 1}| (\mathcal{U}_2)_0^{n- 1}):
n\in \mathbb{N}\}$ is sub-additive, and so we may set
\begin{eqnarray*}
h (T, \mathcal{U}_1| \mathcal{U}_2)= \lim_{n\rightarrow +\infty}
\frac{1}{n} \log N ((\mathcal{U}_1)_0^{n- 1}| (\mathcal{U}_2)_0^{n-
1})\ \left(= \inf_{n\in \mathbb{N}} \frac{1}{n} \log N
((\mathcal{U}_1)_0^{n- 1}| (\mathcal{U}_2)_0^{n- 1})\right).
\end{eqnarray*}
Define the {\it topologically conditional $\mathcal{U}_2$-entropy of
$(X, T)$} by $$h^* (T, X| \mathcal{U}_2)= \sup_{\mathcal{U}_1\in
\mathcal{C}_X^o} h (T, \mathcal{U}_1| \mathcal{U}_2),$$ and the {\it
topologically conditional entropy of $(X, T)$} by $$h^* (T, X)=
\inf_{\mathcal{U}_2\in \mathcal{C}_X^o} h^* (T, X| \mathcal{U}_2).$$
In particular, $h^* (T, X)\le h_{\text{top}} (T, X)$. Observe that, if $(X, T)$
is zero-dimensional then by a standard construction we can represent
$(X, T)$ as an inverse limit of sub-shifts over finite alphabets:
\begin{eqnarray*}
(X, T)= \underleftarrow{\lim} (X_r, T_r)
\end{eqnarray*}
where $X_r\subseteq \Lambda_r^\mathbb{Z}$ with $\Lambda_r$ a finite
discrete space, $T_r$ is the full shift over $\Lambda_r^\Z$ and $(X_r,
T_r)$ is a factor of $(X_{r+ 1}, T_{r+ 1})$ for each $r\in
\mathbb{N}$. Now let $\phi_r: (X, T)\rightarrow (X_r, T_r)$ be the
natural homomorphism and $\mathcal{U}_r$ the clopen generated
partition of $X_r$ ($r\in \mathbb{N}$). Observe that the sequence
$\{h_{\text{top}} (T, X| \phi_r): r\in \mathbb{N}\}$ decreases,
there are some easy but useful facts:
\begin{enumerate}

\item $h_{\text{top}} (T_r, X_r)= h_{\phi_r^{- 1} (\mathcal{U}_r)}
(T, X)$ and $h_{\text{top}} (T, X)= \lim_{r\rightarrow +\infty}
h_{\text{top}} (T_r, X_r)$;

\item $h_{\text{top}} (T, X| \phi_r)\le h^* (T, X| \phi_r^{- 1}
(\mathcal{U}_r))$ and

\item $h^* (T, X)= \lim_{r\rightarrow +\infty} h^* (T, X|
\phi_r^{- 1} (\mathcal{U}_r))\ge \lim_{r\rightarrow +\infty}
h_{\text{top}} (T, X| \phi_r)$.
\end{enumerate}

Thus we have the following interesting result.

\begin{thm} \label{interpret}
Let $(X, T)$ be a t.d.s. with finite entropy and $K\subseteq X$
compact. If $h (T, K)> h^* (T, X)$ then for each $0\le h\le h (T,
K)- h^* (T, X)$ one has
\begin{eqnarray*}
\overline{\{h (T, K'): K'\subseteq K\ \text{is compact}\}}\cap [h,
h+ h^* (T, X)]\neq \emptyset.
\end{eqnarray*}
\end{thm}
\begin{proof}
It's well-known that each t.d.s. with finite entropy has a
zero-dimensional principal extension \cite[Proposition 7.8]{BD},
i.e. an extension preserving entropy for each invariant measure; and
if $\pi$ is a principal extension of a system with finite entropy,
then $\pi$ preserves the topologically conditional entropy
\cite[Theorem 3]{Le} and has zero relative topological entropy by
conditional variational principles \cite[Theorems 3 and 4]{DS}.
Thus, using Proposition \ref{0905191311} (1) we may assume that $(X,
T)$ is zero-dimensional.

 Moreover, it
makes no difference to say $h\in [0, h (T, K)- h^* (T, X))$.

We represent $(X, T)$ by an inverse limit of sub-shifts over finite
alphabets
\begin{eqnarray*}
(X, T)= \underleftarrow{\lim} (X_r, T_r),
\end{eqnarray*}
with $\phi_r: (X, T)\rightarrow (X_r, T_r)$ the natural homomorphism
for each $r\in \mathbb{N}$. Then
\begin{equation*}
h< h (T, K)- h^* (T, X)\le h (T, K)- \lim_{r\rightarrow +\infty}
h_{\text{top}} (T, X| \phi_r)\ (\text{using fact (3)}).
\end{equation*}
For each $r\in \mathbb{N}$ we set $K_r= \phi_r (K)$, so $h (T, K)-
h_{\text{top}} (T, X| \phi_r)\le h (T_r, K_r)\le h (T, K)$ (using
Proposition \ref{0905191311}
(1) again). Thus if $r\in \mathbb{N}$ is large
enough then there exists compact $K^h_r\subseteq K_r$ with $h (T_r,
K^h_r)= h$ (using Theorem \cite[Theorem 5.4]{HYZ1}). Last, put
$K_r^{(h)}= \phi_r^{- 1} (K_r^h)\cap K$. As $\phi_r (K_r^{(h)})=
K_r^h$, by Proposition \ref{0905191311} (1) one has
\begin{eqnarray*}
h= h (T_r, K_r^h)\le h (T, K_r^{(h)})\le h (T_r, K_r^h)+
h_{\text{top}} (T, X| \phi_r)= h+ h_{\text{top}} (T, X| \phi_r).
\end{eqnarray*}
We can claim the conclusion by fact (3).
\end{proof}


\section{A positive answer to Questions \ref{q1} and \ref{q4}}

In this section we shall give a positive answer to Questions
\ref{q1} and \ref{q4}.

Let $(X, T)$ be a t.d.s. Denote by $\mathcal{M} (X)$ (resp.
$\mathcal{M} (X, T)$, $\mathcal{M}^e (X, T)$) the set of all Borel
probability measures (resp. $T$-invariant Borel probability
measures, ergodic $T$-invariant Borel probability measures) on $X$.
All of them are equipped with the weak star topology. Denote by
$\mathcal{B}_X$ the set of all Borel subsets of $X$.

Before proceeding, we need restate \cite[Lemma 4.3 (2)]{HYZ1} as
follows (for the detailed introduction of the (relative)
measure-theoretic entropy and the disintegration of a measure over a
sub-$\sigma$-algebra see for example \cite[\S 4]{HYZ1}).

\begin{lem} \label{0905081904}
Let $(X, T)$ be a t.d.s., $\mu\in \mathcal{M}^e(X, T)$ and
$\mathcal{C}\subseteq \mathcal{B}_\mu$ a $T$-invariant
sub-$\sigma$-algebra (i.e. $T^{- 1} \mathcal{C}= \mathcal{C}$ in the
sense of $\mu$), here $\mathcal{B}_\mu$ is the completion of
$\mathcal{B}_X$ under $\mu$. If $\mu= \int_X \mu_x d \mu (x)$ is the
disintegration of $\mu$ over $\mathcal{C}$, then, for $\mu$-a.e.
$x\in X$, fixing each $x$, for each $\epsilon\in (0, 1)$ there
exists a compact subset $Z_x (\epsilon)$ of $X$ such that $\mu_x
(Z_x (\epsilon))\ge 1- \epsilon$ and $h^B (T, Z_x (\epsilon))= h (T,
Z_x (\epsilon))= h_\mu (T, X| \mathcal{C})$.
\end{lem}

We also need state a relative version of the well-known Sinai
Theorem, which is essentially found in \cite{Orn}. It was made
explicit in \cite[Theorem 5]{OW} and \cite{Th} (for another
treatment of it see \cite{Ki}). Before stating it, we need
make some preparations.

Recall that {\it a $k$-element distribution} ${\bf I}$ is a
probability vector $(I_1, \cdots, I_k)$, i.e. $I_1, \cdots, I_k$ $\ge 0$ and $I_1+ \cdots+ I_k= 1,
k\in \N$. The {\it entropy of it} is defined by
$$H ({\bf I})= \sum_{i= 1}^k - I_i \log I_i.$$
From now on, for a
given t.d.s. $(X, T)$ and $\mu\in \mathcal{M} (X)$, each $\alpha\in
\mathcal{P}_X$ is ordered and associated with a distribution
$\text{dist} \alpha$ (i.e. $\alpha= (A_1, \cdots, A_k)$ is equipped
with a fixed order and in this case $\text{dist} \alpha= (\mu(A_1),
\cdots, \mu(A_k))$).

A relative version of the well-known Sinai Theorem is stated as
follows.

\begin{lem} \label{Orn}
Let $(X, T)$ be a t.d.s., $\mu\in \mathcal{M}^e (X ,T)$ and
$\alpha\in \mathcal{P}_X, \gamma\in \mathcal{P}_X$ with
$\alpha\subseteq \bigvee_{i= - \infty}^{+ \infty} T^{- i} \gamma$
(in the sense of $\mu$). Then, for each $k$-element distribution
${\bf I}, k\in \N$ with $H ({\bf I})\le h_\mu (T, \gamma)- h_\mu (T,
\alpha)$,
 there exists $\beta\in \mathcal{P}_X$ satisfying that
\begin{enumerate}

\item $\beta\subseteq
 \bigvee_{i= - \infty}^{+ \infty} T^{- i} \gamma$ and $\text{dist} \beta ={\bf I}$;

\item the partitions $T^i \beta$, $i\in
\mathbb{Z}$ are independent (in the sense of $\mu$), that is, if
$B_{i_1}\in T^{i_1} \beta, i_1\in \Z$ and $B_{i_2}\in T^{i_2} \beta,
i_2\in \Z$ with $i_1\neq i_2$ then $\mu (B_{i_1}\cap B_{i_2})= \mu
(B_{i_1}) \mu (B_{i_2})$;

\item $\bigvee_{i= - \infty}^{+ \infty} T^{- i} \beta$ is independent
of $\bigvee_{i= - \infty}^{+ \infty} T^{- i} \alpha$ (in the sense
of $\mu$).
\end{enumerate}
\end{lem}

Thus, we have

\begin{prop} \label{Huang}
Let $(X, T)$ be a t.d.s. and $\mu\in \mathcal{M}^e (X, T)$. Then,
for each $0\le h\le h_\mu (T, X)$, there exists a $T$-invariant
sub-$\sigma$-algebra $\mathcal{C}\subseteq \mathcal{B}_\mu$ (in the
sense of $\mu$) with $h_\mu (T, X| \mathcal{C})= h$, here
$\mathcal{B}_\mu$ is the completion of $\mathcal{B}_X$ under $\mu$.
\end{prop}
\begin{proof}
When $h= h_\mu (T, X)$, we may take $\mathcal{C}= \{\emptyset, X\}$,
and so $h_\mu (T, X| \mathcal{C})= h$.

Now we assume that $0\le h< h_\mu (T, X)$. It is not hard to see
that we can take $\alpha\in \mathcal{P}_X$ with $h_\mu (T, \alpha)=
h$. Moreover, by Rokhlin Theorem about countable generators
\cite[10.13]{Rok} there exists a countable measurable partition $\gamma=
\{C_1, C_2, \cdots\}$ (i.e. there exists a sequence of partitions
$\{\alpha_n: n\in \N\}\subseteq \mathcal{P}_X$ with $\gamma=
\alpha_1\vee \alpha_2\vee \cdots\doteq \{A_1\cap A_2\cap \cdots:
A_n\in \alpha_n, n\in \N$\}) where $\bigvee_{i= - \infty}^{+ \infty}
T^{- i} \gamma= \mathcal{B}_\mu$ in the sense of $\mu$.

Set $\eta_0= \{X\}, \gamma_0= \eta_0\vee \alpha$ and $\eta_n= \{C_1,
C_2, \cdots, C_n, X\setminus \bigcup_{i= 1}^n C_i\}, \gamma_n=
\eta_n\vee \alpha, n\in \mathbb{N}$. Then, for each $n\in
\mathbb{N}$, by Lemma \ref{Orn}, there exists $\beta_n\in
\mathcal{P}_X$ with $\beta_n\subseteq \bigvee_{i= - \infty}^{+
\infty} T^{- i} \gamma_n$ (in the sense of $\mu$) such that
\begin{enumerate}

\item the partitions $T^i \beta_n$, $i\in \mathbb{Z}$ are
independent (in the sense of $\mu$);

\item $\bigvee_{i= - \infty}^{+ \infty} T^{- i} \beta_n$ is independent of
$\bigvee_{i= - \infty}^{+ \infty} T^{- i} \gamma_{n- 1}$ (in the
sense of $\mu$) and

\item $H (\text{dist} \beta_n)= h_\mu (T, \gamma_n)- h_\mu (T, \gamma_{n- 1})$.
\end{enumerate}
From (2), one has
\begin{eqnarray}\label{eq-ke1}
h_\mu (T, \gamma_{n-1}| \bigvee_{i= - \infty}^{+ \infty} T^{- i}
\beta_n)= h_\mu (T, \gamma_{n- 1}).
\end{eqnarray}
Moreover, observe that $\gamma_n\vee \beta_n\subseteq \bigvee_{i= -
\infty}^{+ \infty} T^{- i} \gamma_n$ (in the sense of $\mu$), using
the relative Pinsker formula (see for example \cite[Lemma 1.1]{GTW}
or \cite[Theorem 3.3]{Z}) we have

\begin{equation}\label{eq-ke2}
h_\mu (T, \gamma_n\vee \beta_n)= h_\mu (T, \gamma_n),
\end{equation}
and
\begin{eqnarray} \label{eq-keke}
& & h_\mu (T, \gamma_n| \bigvee_{i= - \infty}^{+ \infty} T^{- i}
\beta_n\vee \bigvee_{i= - \infty}^{+ \infty}
T^{- i} \gamma_{n- 1})\nonumber \\
&= & h_\mu (T, \gamma_n| \bigvee_{i= - \infty}^{+ \infty} T^{- i}
\beta_n)- h_\mu (T, \gamma_{n- 1}| \bigvee_{i= - \infty}^{+ \infty}
T^{- i} \beta_n)\ (\text{as}\ \gamma_n\vee \gamma_{n- 1}= \gamma_n)\nonumber \\
&= & h_\mu (T, \gamma_n| \bigvee_{i= - \infty}^{+ \infty} T^{- i} \beta_n)-
h_\mu (T, \gamma_{n- 1})\ (\text{by \eqref{eq-ke1}})\nonumber \\
&= & (h_\mu (T, \gamma_n\vee
\beta_n)- h_\mu (T, \beta_n))- h_\mu (T, \gamma_{n- 1})\nonumber \\
&= & h_\mu (T, \gamma_n)- h_\mu (T, \gamma_{n- 1})- h_\mu (T, \beta_n)\
(\text{by \eqref{eq-ke2}})\nonumber \\
&= & H (\text{dist} \beta_n)- h_\mu (T, \beta_n)\ (\text{by
(3)}) =0\ (\text{by (1)}).
\end{eqnarray}

Put $\mathcal{C}= \bigvee_{n= 1}^{+ \infty} \bigvee_{i= - \infty}^{+
\infty} T^{- i} \beta_n$. Note that, for each $k\in \N$, in the
sense of $\mu$, $\bigvee_{i= - \infty}^{+ \infty} T^{- i} \beta_{k+
1}$ is independent of $\bigvee_{i= - \infty}^{+ \infty} T^{- i}
\gamma_k$ and
$$\bigvee_{j= 1}^k \bigvee_{i= - \infty}^{+ \infty}
T^{- i} \beta_{j}\vee \bigvee_{i= - \infty}^{+ \infty} T^{- i}
\alpha\subseteq \bigvee_{j= 1}^k \bigvee_{i= - \infty}^{+ \infty}
T^{- i} \gamma_{j}= \bigvee_{i= - \infty}^{+ \infty} T^{- i}
\gamma_{k},$$
one has $\bigvee_{i=- \infty}^{+ \infty} T^{- i}
\beta_{k+ 1}$ is independent of $\bigvee_{j= 1}^k \bigvee_{i= -
\infty}^{+ \infty} T^{- i} \beta_j\vee \bigvee_{i= - \infty}^{+
\infty} T^{- i} \alpha$. Combing this with the fact that
$\bigvee_{i= - \infty}^{+ \infty} T^{- i} \beta_1$ is independent of
$\bigvee_{i= -\infty}^{+ \infty} T^{- i} \alpha$ (in the sense of
$\mu$), we have that $\bigvee_{j= 1}^k \bigvee_{i= - \infty}^{+
\infty} T^{- i} \beta_j$ is independent of $\bigvee_{i= - \infty}^{+
\infty} T^{- i} \alpha$ (in the sense of $\mu$) for each $k\in
\mathbb{N}$, and so $\mathcal{C}$ is independent of $\bigvee_{i= -
\infty}^{+ \infty} T^{- i} \alpha$ in the sense of $\mu$.

Finally, we claim that $\mathcal{C}$ is just the
sub-$\sigma$-algebra we need. Obviously, $T^{- 1} \mathcal{C}=
\mathcal{C}$. Now we are going to show $h_\mu (T, X| \mathcal{C})=
h$. On one hand,
$$h_\mu (T, X| \mathcal{C})\ge h_\mu (T, \alpha| \mathcal{C})=
h_\mu (T, \alpha)= h,$$
where the last identity follows from the fact
 that $\mathcal{C}$ is independent of $\bigvee_{i= - \infty}^{+
 \infty}
T^{- i} \alpha$. On the other hand, for each $n\in \mathbb{N}$ by
the relative Pinsker formula
\begin{eqnarray*}
h_\mu (T, \gamma_n| \mathcal{C})&= & h_\mu (T, \alpha| \mathcal{C})+
h_\mu (T, \eta_n| \mathcal{C}\vee \bigvee_{i= - \infty}^{+
\infty} T^{- i} \alpha) \\
&\le & h+ h_\mu (T, \eta_n| \bigvee_{j= 1}^n \bigvee_{i= -
\infty}^{+ \infty} T^{- i} \beta_j\vee
\bigvee_{i= - \infty}^{+ \infty} T^{- i} \alpha) \\
&= & h+ h_\mu (T, \eta_{n}| \bigvee_{j= 1}^n \bigvee_{i= -
\infty}^{+ \infty} T^{- i} \beta_j\vee \bigvee_{i= - \infty}^{+
\infty} T^{- i} \alpha\vee \bigvee_{i= - \infty}^{+ \infty} T^{- i}
\eta_{n- 1})+ \\
& & h_\mu (T, \eta_{n- 1}| \bigvee_{j= 1}^n \bigvee_{i= - \infty}^{+
\infty} T^{- i} \beta_j\vee \bigvee_{i= - \infty}^{+ \infty} T^{- i}
\alpha)\ (\text{as}\ \eta_n\vee \eta_{n- 1}= \eta_n).
\end{eqnarray*}
By \eqref{eq-keke}
\begin{eqnarray*}
h_\mu (T, \gamma_n| \mathcal{C})&\le & h+ h_\mu (T, \eta_{n- 1}|
\bigvee_{j= 1}^n \bigvee_{i= - \infty}^{+ \infty} T^{- i}
\beta_j\vee \bigvee_{i= - \infty}^{+ \infty} T^{- i} \alpha)\
 \\
&\le & h+ h_\mu (T, \eta_{n- 1}| \bigvee_{j= 1}^{n- 1} \bigvee_{i= -
\infty}^{+ \infty} T^{- i} \beta_j\vee
\bigvee_{i= - \infty}^{+ \infty} T^{- i} \alpha) \\
&\le & \cdots\le h+ h_\mu (T, \eta_1| \bigvee_{i= -
\infty}^{+ \infty} T^{- i} \beta_1\vee
\bigvee_{i= - \infty}^{+ \infty} T^{- i} \alpha)= h,
\end{eqnarray*}
which implies
$$h_\mu (T, X| \mathcal{C})= \lim_{n\rightarrow +\infty}
h_\mu (T, \bigvee_{i= - n}^n T^{- i} \gamma_n| \mathcal{C})=
\lim_{n\rightarrow +\infty} h_\mu (T, \gamma_n| \mathcal{C})\le h$$
(as $\gamma_1\preceq \gamma_2\preceq \cdots$ and $\bigvee_{n\in
\mathbb{N}} \bigvee_{i= - n}^n T^{- i} \gamma_n= \mathcal{B}_\mu$ (in the
sense of $\mu$)) and so $h_\mu (T, X| \mathcal{C})= h$. This ends the
proof of Proposition.
\end{proof}

Now we can answer Questions \ref{q1} and \ref{q4} affirmatively.

\begin{thm} \label{0905112120}
Each t.d.s. $(X, T)$ is not only lowerable but also D-lowerable. In
fact, for each $0\le h\le h (T, X)$ there exists compact
$K_h\subseteq X$ with $h (T, K_h)= h^B (T, K_h)= h$.
\end{thm}
\begin{proof}
When $h= h (T, X)$, we may take $K_h= X$. When $h< h (T, X)$, by the
classical variational principle (see for example \cite[Corollary 8.6.1]{Wa}) we may
take $\mu\in \mathcal{M}^e (X, T)$ with $h_\mu (T, X)\ge h$, then by
Proposition \ref{Huang} there exists a $T$-invariant
sub-$\sigma$-algebra $\mathcal{C}\subseteq \mathcal{B}_\mu$ with
$h_\mu (T, X| \mathcal{C})= h$, where $\mathcal{B}_\mu$ is the
completion of $\mathcal{B}_X$ under $\mu$, and so there exists
compact $K_h\subseteq X$ with $h (T, K_h)= h^B (T, K_h)= h_\mu (T,
X| \mathcal{C})= h$ (using Lemma \ref{0905081904}).
\end{proof}

\begin{rem}
We should remark that in \cite{SW} Shub and Weiss presented a t.d.s.
with infinite entropy such that its each non-trivial factor has
infinite entropy.
\end{rem}



\section{A partial answer to Question \ref{q5}}

In this section, we shall give a partial answer to Question \ref{q5}
by proving that each asymptotically $h$-expansive (equivalently,
hereditarily uniformly lowerable) t.d.s. is D-hereditarily
lowerable. 

\subsection{Each asymptotically
$h$-expansive t.d.s. is D-hereditarily lowerable} Recall that, for a
given t.d.s. $(X, T)$ with a compatible metric $d$, $(X, T)$ is
called {\it asymptotically $h$-expansive} if
$\lim_{\epsilon\rightarrow 0+} h_T^*(\epsilon)=0$, where $\epsilon>
0$ and
\begin{equation}\label{esm}
h_T^*(\epsilon)= \sup_{x\in X} h (T, \Phi_{\epsilon} (x)), \
\text{with}\ \Phi_{\epsilon} (x)= \{y\in X: d(T^n x, T^n y)\le
\epsilon\ \text{if}\ n\ge 0\}.
\end{equation}
Observe that it holds $h^* (T, X)= \lim_{\epsilon\rightarrow 0+} h_T^*(\epsilon)$ (\cite[\S 4]{BFF}).

\begin{lem} \label{l1}
Let $(X, T)$ be a symbolic t.d.s. Then $(X, T)$ is not only
hereditarily lowerable but also D-hereditarily lowerable.
\end{lem}
\begin{proof}
Note that any asymptotically
$h$-expansive t.d.s. is hereditarily lowerable (\cite[Theorem 7.7]{HYZ1}), and so  $(X, T)$ is
hereditarily lowerable.

Now let $(X,T)$ be a subshift of
$(\{1,2,\cdots,m\}^{\mathbb{Z}},\sigma)$, where $m\ge 2$ and
$\sigma$ is the left shift on $\{1,2,\cdots,m\}^{\mathbb{Z}}$. Let
$\pi: \{1,2,\cdots,m\}^{\mathbb{Z}}\rightarrow \mathbb{T}$ with
$$\pi((x_j)_{j\in \mathbb{Z}})=e^{2\pi i (\sum_{j=0}^{+\infty}
\frac{x_j-1}{m^{j+1}})}$$ for $(x_j)_{j\in \mathbb{Z}}\in
\{1,2,\cdots,m\}^{\mathbb{Z}}$. Then
$\pi:(\{1,2,\cdots,m\}^{\mathbb{Z}},\sigma)\rightarrow
(\mathbb{T},T_m)$ is a factor map, where $T_m$ is defined by
$T_m(z)=z^m$ for $z\in \mathbb{T}$.

For each $z\in \mathbb{T}$, there are $y^1,y^2\in
\{1,2,\cdots,m\}^{\mathbb{Z}_+}$  such that
$$\pi^{-1}(z)=\{x=(x_j)_{j\in \mathbb{Z}}\in
\{1,2,\cdots,m\}^{\mathbb{Z}}: (x_j)_{j=0}^{+\infty}=y^1 \text{ or
}y^2\},$$ here in fact for almost all $z\in \mathbb{T}$, $y^1=y^2$.
Thus $h^B(T,\pi^{-1}(z))\le h(T,\pi^{-1}(z))=0$, and so
$h_{\text{top}}(\sigma,\{1,2,\cdots,m\}^{\mathbb{Z}}|\pi)=0$. Hence
by Proposition \ref{0905191311},
\begin{equation}\label{eeee-11}
h^B(\sigma,K)=h^B(\pi(K),T_m)
\end{equation}
for any $K\subseteq \{1,2,\cdots,m\}^{\mathbb{Z}}$. Since any
continuous image of a Souslin set in $\{1,2,\cdots,m\}^{\mathbb{Z}}$
is Souslin, we know that $(\{1,2,\cdots,m\}^{\mathbb{Z}},\sigma)$ is
D-hereditarily lowerable by Proposition \ref{L-exapnsive} (2) and
\eqref{eeee-11}. As a subsystem of
$(\{1,2,\cdots,m\}^{\mathbb{Z}},\sigma)$, $(X,T)$ is also
D-hereditarily lowerable.
\end{proof}

\begin{thm} \label{again}
Each asymptotically
$h$-expansive t.d.s. is D-hereditarily lowerable.
\end{thm}
\begin{proof}
Observe that each asymptotically $h$-expansive t.d.s. admits a
principal extension to a symbolic t.d.s. (\cite[Theorem 8.6]{BD}).
As commented in the first paragraph of proof of Theorem
\ref{interpret}, a principal extension has zero relative topological
entropy, and so by Proposition \ref{0905191311} and Lemma \ref{l1}
we obtain the conclusion.
\end{proof}

Remark that, by the same argument presented in the proof of Theorem
\ref{interpret}, we can obtain the following result with the help of
Proposition \ref{0905191311} (2) and Lemma \ref{l1}.

\begin{prop}
Let $(X, T)$ be a t.d.s. with finite entropy and $K\subseteq X$ Souslin. If $h^B (T, K)> h^* (T, X)$ then for each $0\le h\le h^B (T,
K)- h^* (T, X)$ one has
\begin{eqnarray*}
\overline{\{h^B (T, K'): K'\subseteq K\}}\cap [h,
h+ h^* (T, X)]\neq \emptyset.
\end{eqnarray*}
\end{prop}

The property of asymptotical $h$-expansiveness can be weakened
respectively as follows.

\begin{defn}
Let $(X, T)$ be a t.d.s. with a compatible metric $d$. We call it
\begin{enumerate}

\item {\it anti-asymptotically $h$-expansive}, if for
each $\epsilon> 0$ there exists a factor map $\pi: (Y, S)\rightarrow
(X, T)$, where $(Y, S)$ is a symbolic t.d.s., such that
$h_{\text{top}} (S, Y| \pi)< \epsilon$;

\item {\it quasi-asymptotically $h$-expansive}, if
 $\lim_{\epsilon\rightarrow 0+} h (T, \Phi_{\epsilon} (x))= 0$
for each $x\in X$.

\end{enumerate}
\end{defn}

We remark that there is a D-hereditarily lowerable t.d.s. which is
not quasi-asymptotically $h$-expansive. Let $(X_1, T_1)$ be the full
shift over $\Sigma_2= \{0, 1\}^\Z$ and $(X_2, T_2)$ the identity
transformation of the one point compactification over $\Z$ with
$X_2= \Z\cup \{\infty\}$. Put $(Y, S)= (X_1, T_1)\times (X_2, T_2)$.
Collapsing $X_1\times \{\infty\}$ to one point $x_0$ we get a t.d.s.
$(X, T)$. It is not hard to check that $\lim_{\ep\rightarrow 0+} h
(T, \Phi_\ep (x_0))= \log 2$, which implies that t.d.s. $(X, T)$ is
not quasi-asymptotically $h$-expansive, whereas, clearly $(X, T)$ is
D-hereditarily lowerable by Theorem \ref{again}.

In fact, we can construct a D-hereditarily lowerable t.d.s. which is
not asymptotically $h$-expansive but quasi-asymptotically
$h$-expansive.

\begin{ex} \label{0905221459}
There is a D-hereditarily lowerable t.d.s. which is not
asymptotically $h$-expansive but quasi-asymptotically $h$-expansive.
\end{ex}
\begin{proof}
For each $n\in \N$ we choose $\epsilon_n> 0$ and $C_n\subseteq [0,
1]$ homeomorphic to the Cantor set with $\lim_{n\rightarrow +\infty}
\epsilon_n= 0$ and $\lim_{n\rightarrow +\infty} C_n= [0, 1]$ (in the
sense of Hausdorff metric). Set $X= X_0\cup X_1\cup X_2\cup \cdots$,
where $X_0= \{0\}\times [0, 1]\subseteq \R^2$ and $X_n=
\{\frac{1}{n}\}\times C_n\subseteq \R^2$ for each $n\in \N$. Now for
each $n\in \N$ we let $T_n: X_n\rightarrow X_n$ be a minimal
sub-shift such that $h (T_n, \Phi_{\epsilon_n} (x_n))\ge \log 2$ for
some $x_n\in X_n$ (we may assume that $\lim_{n\rightarrow +\infty}
x_n= x_0$) and let $T_0: X_0\rightarrow X_0$ be the identity map.
Last, $(X, T)$ is defined naturally. We may add the assumptions on
the defined $(X_n, T_n), n\in \N$ such that $(X, T)$ forms a t.d.s.

We claim that t.d.s. $(X, T)$ is the system we need. It is not hard
to check that $(X, T)$ is a D-hereditarily lowerable t.d.s.
Clearly, $\lim_{\epsilon\rightarrow 0+} h (T, \Phi_\epsilon (x))= 0$
for each $x\in X$ and so $(X, T)$ is quasi-asymptotically
$h$-expansive; whereas, $$\lim_{\ep\rightarrow 0+} \sup_{x\in X} h
(T, \Phi_\ep(x))\ge \limsup_{n\rightarrow +\infty} h (T_n,
\Phi_{\epsilon_n} (x_n))\ge \log 2,$$ that is, t.d.s. $(X, T)$ is
not asymptotically $h$-expansive.

To finish our example it remains to construct a t.d.s. $(X_n, T_n)$
and $\epsilon_n> 0$ as above for each $n\in \N$. Let $n\in \N$. For
each $j= 1, \cdots, 2 n$ we put $C_n^j= \frac{2 j}{4 n+ 1}+
C_n^0\subseteq J_n^j= [\frac{2 j}{4 n+ 1}, \frac{2 j+ 1}{4 n+ 1}]$
with $C_n^0$ being linearly homeomorphic to the standard Cantor set
and set $C_n= \bigcup_{j= 0}^{2 n} C_n^j$. Define a permutation
$\phi_n: \{0, 1, \cdots, 2 n\}\rightarrow \{0,1,\cdots, 2n\}$ such
that $\phi_n^j (i)= i, 1\le j\le 2 n+ 1$ implies $j= 2 n+ 1$ and
$|\phi_n (i)- i|\le 2$ for each $0\le i\le 2 n$. For example, we set
$\phi_n(0)=2,\phi_n(2)=4, \cdots, \phi_n(2n-2)=2n, \phi_n(2n)=2n-1$;
$\phi_n(2n-1)=2n-3, \phi_n(2n-3)=2n-5, \cdots, \phi_n(3)=1$ and
$\phi_n(1)=0$. Now let $(C_n^0,S_n)$ be a minimal sub-shift with $h
(S_n, C_n^0)\ge (2n+1)\log 2$. We define $T_n:X_n\rightarrow X_n$ as
follows: $T_n(\frac{1}{n}, y_n)=(\frac{1}{n},
\frac{2\phi_n(j)}{4n+1}+c_n)$ if $y_n\not\in C_n^{\phi_n^{- 1} (0)}$
with $y_n=\frac{2j}{4n+1}+c_n$, where $c_n\in C_n^0$, and
$T_n(\frac{1}{n},y_n)=(\frac{1}{n}, S_n(c_n))$ if $y_n=\frac{2
\phi_n^{- 1} (0)}{4n+1}+c_n\in C_n^{\phi_n^{- 1} (0)}$, where
$c_n\in C_n^0$. Let $\epsilon_n=\frac{5}{4n+1}$. It is easy to check
that the constructed $(X_n, T_n)$ is a minimal sub-shift. Moreover,
for each $x_n\in C_n^0$, $C_n^0\subseteq \Phi_{\epsilon_n} (x_n)$
and so (by Lemma \ref{prod0} (1))
\begin{eqnarray*}
h (T_n, \Phi_{\epsilon_n} (x_n))&= & \frac{1}{2 n+ 1} h (T_n^{2 n+
1}, \Phi_{\epsilon_n} (x_n)) \\
&\ge & \frac{1}{2 n+ 1} h (T_n^{2 n+ 1}, C_n^0)= \frac{1}{2 n+ 1} h
(S_n, C_n^0)\ge \log 2
\end{eqnarray*}
(in fact, this works for each $x_n\in X_n$). Last, by the assumption
of $\phi_n$ it is not hard to check that the defined $(X, T)$ is a
t.d.s. This finishes the construction of our example.
\end{proof}

\begin{rem} \label{0905221507}
In the above example we can choose $(C_n^0, S_n), n\in \N$ such that
$h_{\text{top}} (T, X)= +\infty$, and so one has that there is a
D-hereditarily lowerable t.d.s. (with entropy infinite) which is not
anti-asymptotically $h$-expansive but quasi-asymptotically
$h$-expansive. We do not know if there is such an example which is
minimal.
\end{rem}

\subsection{Each anti-asymptotically
$h$-expansive t.d.s. is asymptotically $h$-expansive} In this
subsection we will show that asymptotical $h$-expansiveness and
anti-asymptotical $h$-expansiveness are equivalent properties. For
that, we need some notions and results in \cite{BD}.

Given a t.d.s. $(X, T)$, we will say a sequence of partitions
$\{\alpha_k\}$ of $X$ is {\it refining} if the maximum diameter of
elements of $\alpha_k$ goes to zero with $k$; and for each $k$ the
partition $\alpha_{k+1}$ refines $\alpha_k$. The partitions have
{\it small boundaries} if their boundaries have measure zero for all
$\mu$ in $\mathcal{M}(X,T)$. For a finite entropy t.d.s. $(X,T)$
admitting a nonperiodic minimal factor, by \cite[ Theorem 6.2]{L2}
and \cite[Theorem 4.2]{LW} we know that $(X,T)$ has the so called
small boundary property, which is equivalent to the existence of a basis of
the topology consisting of sets whose boundaries have measure zero
for every invariant measure. Moreover, it is easy to construct the
refining sequence of partitions with small boundaries for $(X, T)$ (see \cite[Theorem
7.6 (3)]{BD}).

\begin{defn} Let $(X,T)$ be a finite entropy t.d.s. admitting a
nonperiodic minimal factor. An entropy structure for $(X, T)$  is a
sequence $\mathcal{H}$ of functions $\{h_k\}$ defined on
$\mathcal{M}(X,T)$ in the following way: suppose $\{\alpha_k\}$ is a
refining sequence of finite Borel partitions with small boundaries, then $h_k: \mathcal{M}(X,T)\ra \R $ is obtained by setting
$\mu\mapsto h_\mu(T,\alpha_k)$ for each $k\in \N$.
\end{defn}

The following result follows from \cite[Theorem 8.6]{BD}.
\begin{prop}  \label{asy-ch}Let $(X,T)$ be a  finite entropy t.d.s. admitting a
nonperiodic minimal factor. The following statements are equivalent for $(X,T)$
with entropy structure $\mathcal{H}$:
\begin{enumerate}
\item $(X,T)$ is asymptotically h-expansive.
\item  $h_k$ converges uniformly to the entropy function $h$, where
$h(\mu):=h_{\mu}(T,X)$ for each $\mu\in \mathcal{M}(X,T)$.
\end{enumerate}
\end{prop}

\begin{defn} \label{de1} A function $f : K\rightarrow \R$ defined on a compact metric space
$K$ is upper semicontinuous (u.s.c.) if one of the following
equivalent conditions holds:
\begin{enumerate}
\item $f = \inf_{i\in I} f_i$ for some family $\{f_i\}_{i\in I}$ of continuous functions.

\item $f = \lim_{i\rightarrow +\infty} g_i$, where $\{g_i\}$ is a nonincreasing sequence of continuous functions.

\item For each $r \in \mathbb{R}$, the set $\{x\in K: f(x) \ge  r\}$ is closed.

\item  $\limsup\limits_{x'\rightarrow x} f(x') \le f(x)$ at each $x\in K$.
\end{enumerate}
\end{defn}

For u.s.c functions, the following properties hold:
\begin{itemize}
\item[{a)}] The infimum of any family of u.s.c. functions is
u.s.c. (by Definition \ref{de1} (1)).

\item[{b)}] Both the sum and the supremum of finitely many u.s.c.
functions are u.s.c. (by  Definition \ref{de1} (2)).

\item[{c)}] Every u.s.c. function from a compact metric space to $\R$ is bounded
above and attains its maximum.
\end{itemize}

The following result is \cite[Proposition 2.4]{BD}.
\begin{prop} \label{esi}(Exchanging Suprema and Infima). If $\{f_i\}_{i\in \mathbb{N}}$ is a nonincreasing
sequence of u.s.c. functions on a compact metric space $K$, then
$$\inf_{i\in \mathbb{N}} \sup_{x\in K} f_i(x)= \sup_{x\in K} \inf_{i\in \mathbb{N}} f_i(x).$$
\end{prop}

\begin{thm} \label{qahe}
Asymptotical $h$-expansiveness and anti-asymptotical $h$-expansiveness
are equivalent properties.
\end{thm}
\begin{proof}
Remark that each asymptotically
$h$-expansive t.d.s. admits a principal extension to a symbolic t.d.s. which has zero relative topological entropy (cf the proof of Theorem \ref{again}).
Thus
an asymptotically $h$-expansive
t.d.s. is clearly anti-asymptotically $h$-expansive by definitions.

Conversely, let $(X, T)$ be an anti-asymptotically $h$-expansive
t.d.s. Suppose $(Z,R)$ is an aperiodic minimal zero entropy system.
Let $Y=X\times Z$ and $S=T\times R$. Then $(Y,S)$ is an
anti-asymptotically $h$-expansive t.d.s., since $(Z,R)$ is
asymptotically $h$-expansive.

Now $(Y,S)$ is a finite entropy t.d.s. admitting a nonperiodic
minimal factor $(Z,R)$. Suppose $\{\alpha_k\}$ is a refining sequence of
finite Borel partitions of $(Y,S)$ with small boundaries. Define
$\mathcal{H}$ by setting $h_k :\mu\mapsto h_\mu(S,\alpha_k)$ for each $k\in \N$. Then
$\mathcal{H}$ is an entropy structure of $(Y,S)$.

For $m\in \mathbb{N}$, let $\pi_m: (X_m,T_m)\rightarrow (Y,S)$ be a
factor map such that $(X_m,T_m)$ is a symbolic t.d.s. and
$h_{\text{top}}(T_m, X_m|\pi_m)<\frac{1}{m}$. Then $\pi_m$ induces a
continuous map $\pi_m^*:\mathcal{M}(X_m,T_m)\rightarrow
\mathcal{M}(Y,S)$ satisfying $\pi_m^*\nu(A)=\nu(\pi^{-1}_mA)$ for
any $\nu\in \mathcal{M}(X_m,T_m)$ and  any Borel subset $A$ of $Y$.

Let $\beta_m$ be a generating clopen partition of $(X_m,T_m)$. Now
we consider the function $g^{m}_k:\mathcal{M}(X_m,T_m)\rightarrow
\mathbb{R}$ with
$g^m_{k}(\nu)=h_\nu(T_m,X_m)-h_{\pi_m^*\nu}(S,\alpha_k)$ for each $\nu\in
\mathcal{M}(X_m,T_m)$. Then for $k\in \mathbb{N}$,
\begin{align*}
g_k^m(\nu)&=h_\nu(T_m,\beta_m\vee
\pi_m^{-1}(\alpha_k))-h_{\nu}(T_m,\pi_m^{-1}\alpha_k)\\
&=\lim_{N\rightarrow +\infty}
\frac{1}{N}H_\nu(\bigvee_{i=0}^{N-1}T^{-i}_m (\beta_m\vee
\pi_m^{-1}(\alpha_k))|\bigvee_{i=0}^{N-1}T^{-i}_m\pi_m^{-1}\alpha_k)\\
&=\inf_{N\ge 1} \frac{1}{N} H_\nu(\bigvee_{i=0}^{N-1}T^{-i}_m (\beta_m\vee
\pi_m^{-1}(\alpha_k))|\bigvee_{i=0}^{N-1}T^{-i}_m\pi_m^{-1}\alpha_k).
\end{align*}
The last equality follows from the fact that the sequence
$$a_n(\nu):=H_\nu(\bigvee_{i=0}^{n-1}T^{-i}_m (\beta_m\vee
\pi_m^{-1}(\alpha_k))|\bigvee_{i=0}^{n-1}T^{-i}_m\pi_m^{-1}\alpha_k)$$
is subadditive, i.e. $a_{n_1+n_2}(\nu)\le
a_{n_1}(\nu)+a_{n_2}(\nu)$. Since $\beta_m$ and $\pi_m^{-1}\alpha_k$
have small boundaries, $\nu\mapsto H_\nu(\bigvee_{i=0}^{N-1}T^{-i}_m
(\beta_m\vee
\pi_m^{-1}(\alpha_k))|\bigvee_{i=0}^{N-1}T^{-i}_m\pi_m^{-1}\alpha_k)$
is a continuous function on $\mathcal{M}(X_m,T_m)$ for each $N\in
\N$. Thus the function $g_k^m$ is u.s.c. by Definition \ref{de1}
(1).

Next, we let $f_k^m: \mathcal{M}(Y,S)\rightarrow \mathbb{R}$ with
\begin{align*}
f_k^m(\mu):&=\sup\{ g_k^m(\nu):\nu \in
(\pi_m^*)^{-1}(\mu)\}\\
&=\sup\{ h_\nu(T_m,X_m):\nu \in
(\pi_m^*)^{-1}(\mu)\}-h_{\mu}(S,\alpha_k)
 \end{align*} for each $\mu\in
\mathcal{M}(Y,S)$. Since $g_k^m$ is u.s.c., it is easily seen using
Definition \ref{de1} (4) that $f_k^m$ is also u.s.c.

Now let $f_k(\mu):=h_{\mu}(S,Y)-h_{\mu}(S,\alpha_k)$ for each $\mu\in
\mathcal{M}(Y,S)$. Then $\{f_k\}$ is a decreasing sequence of
non-negative functions on $\mathcal{M}(Y,S)$ and $\lim_{k\rightarrow
+\infty}f_k(\mu)=0$ for all $\mu\in \mathcal{M}(Y,S)$. Since
\begin{align*}
f_k^m(\mu)-f_k(\mu)&=\sup\{ h_\nu(T_m,X_m)-h_{\pi_m^*\nu}(S,Y):\nu
\in (\pi_m^*)^{-1}(\mu)\}\\ &=\sup\{ h_\nu(T_m,X_m|\pi_m):\nu \in
(\pi_m^*)^{-1}(\mu)\}\\
& \in [0,h_{\text{top}}(T_m,X_m|\pi_m)]\subseteq [0,\frac{1}{m}]
\end{align*}
for any $\mu\in \mathcal{M}(Y,S)$,  we have $f_k=\inf\limits_{m\in
\mathbb{N}}f_k^m$. Thus $f_k$ is u.s.c. since each $f_k^m$ is u.s.c.

By Proposition \ref{esi},
$$\lim_{k\rightarrow +\infty} \sup_{\mu\in
\mathcal{M}(Y,S)}f_k(\mu)=\inf_{k\in \mathbb{N}} \sup_{\mu\in
\mathcal{M}(Y,S)}f_k(\mu)= \sup_{\mu\in \mathcal{M}(Y,S)}\inf_{k\in
\mathbb{N}}f_k(\mu)=0.$$ Thus $h_k$  converges uniformly to the
entropy function $h$, where $h_k(\mu)=h_{\mu}(S,\alpha_k)$ and
$h(\mu)=h_{\mu}(S,Y)$. Hence the system $(Y,S)$ is asymptotically $h$-expansive
by Proposition \ref{asy-ch}.

Finally since for any $x\in X,z\in Z$ and $\epsilon>0$,
$\Phi_\epsilon((x,z))\supset \Phi_\epsilon(x)\times \{z\}$ by
definition \eqref{esm}, we have $h_{T\times R}^*(\epsilon)\ge
h_{T}^*(\epsilon)$. So $\lim_{\epsilon\rightarrow 0+}
h_T^*(\epsilon)=0$ since $\lim_{\epsilon\rightarrow 0+} h_{T\times
R}^*(\epsilon)=0$. Thus $(X,T)$ is asymptotically $h$-expansive.
\end{proof}


\begin{thebibliography}{99}

\bibitem {AG} E. Akin and E. Glasner, {\it Residual properties and almost equicontinuity},
J. Anal. Math., {\bf 84} (2001), 243-286.


\bibitem{B}R. Bowen, \emph{Entropy for group endomorphisms and
homogeneous spaces}, {Trans. Amer. Math. Soc., \bf 153} (1971),
401-414.

\bibitem{Bex} R. Bowen, \emph{Entropy-expansive maps},
{Trans. Amer. Math. Soc., \bf 164} (1972), 323-331.

\bibitem{B3} R. Bowen, \emph{Topological entropy for noncompact sets},
{Trans. Amer. Math. Soc., \bf 184} (1973), 125-136.

\bibitem{BD} M. Boyle and T. Downarowicz, \emph{The entropy theory
of symbolic extensions}, {Invent. Math., \bf 156} (2004), 119-161.

\bibitem{BFF} M. Boyle, D. Fiebig and U. Fiebig,
\emph{Residual entropy, conditional entropy and subshift covers},
{Forum Math., \bf 14} (2002), 713-757.



\bibitem{D2} T. Downarowicz, \emph{Entropy of a symbolic
extension of a dynamical system}, {Ergod. Th. and Dynam. Sys., \bf
21} (2001), 1051-1070.


\bibitem{DL} T. Downarowicz and Y. Lacroix, {\it Almost $1$-$1$ extensions of
Furstenberg-Weiss type and applications to Toeplitz flows}, Studia
Math., {\bf 130} (1998),  149-170.

\bibitem{DS} T. Downarowicz and J. Serafin, \emph{Fiber
entropy and conditional variational principles in compact
non-metrizable spaces}, {Fund. Math., \bf 172} (2002), 217-247.

\bibitem{Fal} K. J. Falconer, \emph{Fractal geometry ---- Mathematical
Foundations and Applications}, 1990.

\bibitem{CHYZ}  C. Fang, W. Huang, Y. Yi and P. Zhang, \emph{Dimension of stable sets and
scrambled sets in positive finite entropy systems}, {Ergod. Th. and
Dynam. Sys.,} {\bf 32} (2012), 599-628.


\bibitem{Fu0} H. Furstenberg, \emph{Disjointness in ergodic theory,
minimal sets, and a problem in Diophantine approximation},
{Mathematical Systems Theory, \bf 1} (1967), 1-49.

\bibitem{Fu} H. Furstenberg, \emph{Recurrence in Ergodic Theory and
Combinatorial Number Theory}, Princeton Univ. Press, 1981.

\bibitem{G} E. Glasner, \emph{Ergodic Theory via Joinings},
{Mathematical Surveys and Monographs \bf 101}, American Mathematical
Society, 2003.

\bibitem{GTW} E. Glasner, J.-P. Thouvenot and B. Weiss, {\it Entropy theory without
a past}. Ergod. Th. and Dynam. Sys., {\bf 20} (2000), 1355-1370.

\bibitem{HYZ} W. Huang, X. Ye and G. H. Zhang, \emph{A local
variational principle for conditional entropy}, {Ergod. Th. and
Dynam. Sys., \bf 26} (2006), 219-245.

\bibitem{HYZ1} W. Huang, X. Ye and G. H. Zhang, \emph{Lowering
topological entropy over subsets}, Ergod. Th. and Dynam. Sys., {\bf
30} (2010), 181-209.

\bibitem{King} J. King, {\it A map with topological minimal self-joinings in
the sense of del Junco}, Ergod. Th. and Dynam. Sys., {\bf 10}
(1990), 745-761.

\bibitem{Ka} A. Katok, \emph{Lyapunov exponents, entropy
and periodic orbits for diffeomorphisms}, {Inst. Hautes Etudes. Sci.
Publ. Math., \bf 51} (1980), 137-173.

\bibitem{Ki} J. C. Kieffer, {\it A simple development of the Thouvenot relative
isomorphism theory}. Ann. Prob., {\bf 12} (1984), 204-211.

\bibitem{KN} H. Keynes and D. Newton, {\it Real prime flows}, Trans. Amer. Math.
Soc., {\bf 217}(1976), 237-255.

\bibitem{Le} F. Ledrappier, \emph{A variational principle
for the topological conditional entropy}, {Lecture Notes in Math.
\bf 729} (1979), Springer-Verlag, 78-88.

\bibitem{LW1} F. Ledrappier and P.
Walters, \emph{A relativised variational principle for continuous
transformations}, {J. London Math. Soc. (2), \bf 16} (1977),
568-576.


\bibitem{L1} E. Lindenstrauss, \emph{Lowering topological entropy},
{J. Anal. Math., \bf 67} (1995), 231-267.

\bibitem{L2} E. Lindenstrauss,
\emph{Mean dimension, small entropy factors and an embedding
theorem}, {Inst. Hautes Etudes. Sci. Publ. Math., \bf 89} (1999),
227-262.

\bibitem{LW} E. Lindenstrauss and B. Weiss,
\emph{Mean topological dimension}, {Israel J. Math., \bf 115}
(2000), 1-24.

\bibitem{M} P. Mattila, \emph{Geometry of Sets and Measures in Euclidean
Spaces, Fractals and Rectifiability}, {Cambridge Studies in Advanced
Mathematics \bf 44}, Cambridge University Press, Cambridge, 1995.

\bibitem{Mi} M. Misiurewicz, \emph{Diffeomorphism without
any measure with maximal entropy}, {Bull. Acad. Pol. Sci., \bf 21}
(1973), 903-910.

\bibitem{Mi1} M. Misiurewicz, \emph{Topological conditional entropy},
{Studia Math., \bf 55} (1976), 175-200.

\bibitem{Mi-add} M. Misiurewicz, \emph{On Bowen's definition of topological entropy},
{Discrete Contin. Dyn. Syst., \bf 10} (2004), 827-833.

\bibitem{Z1} P. Oprocha and G. H. Zhang, \emph{Dimensional entropy over sets and fibres}, {Nonlinearity, \bf 24} (2011), 2325-2346.

\bibitem{Orn} D. S. Ornstein, {\it Bernoulli shifts with infinite entropy are isomorphic},
Adv. Math., {\bf 5} (1970), 339-348.

\bibitem{OW} D. S. Ornstein and B. Weiss, {\it Unilateral codings of Bernoulli
systems}, Israel J. Math., {\bf 21} (1975), 159-166.

\bibitem{P} Y. Pesin, \emph{Dimension Theory in Dynamical Systems:
Contemporary Views and Applications}, Chicago Lectures in
Mathematics, The University of Chicago Press, Chicago, 1998.

\bibitem{Rok} V. A. Rokhlin, {\it Lectures on the entropy theory of transformations
with invariant measure}, Russian Math. Surveys, {\bf 22} (5) (1967),
1-52.

\bibitem{SW} M. Shub and B. Weiss, \emph{Can one always lower
topological entropy}? {Ergod. Th. and Dynam. Sys., \bf 11} (1991),
535-546.

\bibitem{Th} J.-P. Thouvenot, {\it Quelques proprietes des systemes dynamiques
 qui se decomposent en un produit de
deux systemes dont l¡¯un est un schema de Bernoulli}, Israel J.
Math., {\bf 21} (1975), 177-207.

\bibitem{Wa} P. Walters, \emph{An Introduction to Ergodic Theory},
{Graduate Texts in Mathematics \bf 79}, Springer-Verlag, New
York-Berlin, 1982.

\bibitem{Weiss} B. Weiss, {\it Multiple recurrence and doubly minimal systems}, Contemp. Math., {\bf 215},
189-196,  Amer. Math. Soc., Providence, RI, 1998.

\bibitem{YZ} X. Ye and G. H. Zhang, \emph{Entropy
points and applications}, {Trans. Amer. Math. Soc.,} {\bf
359} (2007), 6167-6186.

\bibitem{Z} G. H. Zhang, \emph{Relative entropy, asymptotic pairs and
chaos}, {J. London Math. Soc. (2), \bf 73} (2006), 157-172.


\end{thebibliography}
\end{document}